%% file: SL_1,D_.tex
\documentclass{amsart}
\usepackage{amssymb
,amsthm
,amsmath
,amscd
,mathtools
}

\usepackage[pagewise]{lineno}

\usepackage{appendix}
\usepackage{enumerate}
\usepackage{hyperref}
\usepackage[all]{xy}
\usepackage[top=30truemm,bottom=30truemm,left=25truemm,right=25truemm]{geometry}

\newcommand{\Sp}{\mathrm{Sp}}
\newcommand{\SL}{\mathrm{SL}}
\newcommand{\GL}{\mathrm{GL}}
\newcommand{\Mp}{\mathrm{Mp}}
\newcommand{\Oo}{\mathrm{O}}
\newcommand{\SO}{\mathrm{SO}}

\newcommand{\GSO}{\mathrm{GSO}}
\newcommand{\PGL}{\mathrm{PGL}}

\newcommand{\Hom}{\mathrm{Hom}}
\newcommand{\Gal}{\mathrm{Gal}}

\newtheorem{thm}{Theorem}[section]

\newtheorem{lem}[thm]{Lemma}

\newtheorem{coro}[thm]{Corollary}

\theoremstyle{remark}
\newtheorem{rem}[thm]{Remark}

\theoremstyle{definition}
\newtheorem{defn}[thm]{Definition}

\numberwithin{equation}{section}

\makeatletter
\def\iddots{\mathinner{\mkern1mu\raise\p@
	\hbox{.}\mkern2mu\raise4\p@\hbox{.}\mkern2mu
	\raise7\p@\vbox{\kern7\p@\hbox{.}}\mkern1mu}}
\def\adots{\mathinner{\mkern2mu\raise\p@\hbox{.}
 \mkern2mu\raise4\p@\hbox{.}\mkern1mu
 \raise7\p@\vbox{\kern7\p@\hbox{.}}\mkern1mu}}
\makeatother

\allowdisplaybreaks
\title{The $\SL_1(D)$-distinction problem}
\author{Hengfei LU}
\address{School of Mathematics,	Tata Institute of Fundamental Research, Dr. Homi Bhabha Road,	Colaba, Mumbai 400005, INDIA}
\email{hengfei@math.tifr.res.in}
\begin{document}
\maketitle

\begin{abstract}
We use the local theta correspondences between the quaternionic Hermitian groups and the quaternionic skew-Hermitian groups to understand the distinction problem for the symmetric pair $\SL_2(E)/\SL_1(D)$, where $E$ is a quadratic field extension of a nonarchimedean local field  extension $F$ and $D$ is a $4$-dimensional division quaternion algebra over $F$. 
\end{abstract}
	\subsection*{Keywords} theta lifts, distinction problems, division quaternion algebra
\subsection*{MSC(2000)} 11F27$\cdot$11F70$\cdot$22E50
\tableofcontents
\section{Introduction}
Distinction problems are very popular in representation theory. Let $F$ be a finite field extension over $\mathbb{Q}_p$. Let $G$ be a reductive group defined over $F$.
Let $H$ be a closed subgroup of $G$. Given a smooth represenation $\pi$ of $G(F)$ and a character $\chi_H$ of $H(F)$, if $\dim\Hom_{H(F)}(\pi,\chi_H)$ is nonzero, then $\pi$ is called $(H(F),\chi_H)$-distinguished. Furthermore, if $\chi_H$ is a trivial character, then $\pi$ is called $H(F)$-distinguished. There is a rich literature, such as
\cite{adler2006on,flicker1994quaternionic,prasad2006sl2,prasad2015arelative,prasadcomp}, trying to classify all $H(F)$-distinguished representations of $G(F)$. In this paper, we will focus on the case $G=R_{E/F}\SL_2$, $H=\SL_1(D)$ and $\chi_H$ is trivial, where $E/F$ is a quadratic field extension, $D$ is the unique $4$-dimensional quaternion division algebra defined over $F$ and $R_{E/F}$ denotes the Weil restriction of scalars. 

Let $E$ be a quadratic field extension of a nonarchimedean local field $F$ of charactristic $0$. Let $W_E$ (resp. $W_F$) be the Weil group of $E$ (resp. $F$) and $WD_E$ (resp. $WD_F$)
be the Weil-Deligne group of $E$ (resp. $F$). Let $G$ be a quasi-split reductive group defined over $F$ with Langlands dual group $\hat{G}$. Let $\pi$ be an irreducible smooth representation of $G(E)$ with enhanced Langlands parameter $(\phi_{\pi},\lambda)$, where
\[\phi_{\pi}:WD_E\longrightarrow \hat{G}(\mathbb{C})\rtimes W_E \]
is the Langlands parameter and $\lambda$ is a character of the component group $\pi_0(C_{\hat{G}}(\phi_{\pi}))$, where $C_{\hat{G}}(\phi_{\pi})$ is the centralizer of $\phi_{\pi}$ in $\hat{G}$. 
In \cite{prasad2015arelative},
Dipendra Prasad formulates a conjectural identity for the dimension $\dim\Hom_{G_\alpha(F)}(\pi,\chi_G) $, in terms of the Langlands parameter $\tilde{\phi}$ of $G^{op}$ satisfying $\tilde{\phi}|_{WD_E}=\phi_{\pi}$, where $G^{op}$ is a quasi-split group defined in \cite[\S9]{prasad2015arelative}, $G_\alpha$ is the pure inner form of $G$ and $\chi_G$ is a quadratic character of $G(F)$ defined in \cite[\S10]{prasad2015arelative}.
\par
  It is natural to ask what happens if $G_\alpha$ is an inner form of $G$ satisfying $G_\alpha(E)=G(E)$. There is a well-known result of Dipendra Prasad \cite{Dipendra1992invariant} and Jeffrey Hakim \cite{Hakim1991dist} about $D^\times$-distinguished representation $\pi$ of $\GL_2(E)$.
\begin{thm}\cite[Theorem C]{Dipendra1992invariant}
	Let $\pi$ be a square-integrable representation of $\GL_2(E)$, then $\pi$ is $D^\times$-distinguished if and only if $\pi$ is $\GL_2(F)$-distinguished.
\end{thm}
\begin{rem}
	Raphael Beuzart-Plessis \cite{beuzart2017distinguished}  generalizes this result \cite[Theorem C]{Dipendra1992invariant} to any inner form $G'$ of a quasi-split reductive group $G$ for the stable square-integrable representations. More precisely, let $E$ be a quadratic field extension over a nonarchimedean local field $F$. Given a quadratic character $\chi_{G,E}$ of $G(F)$ and a quadratic character $\chi_{G',E}$ of $G'(F)$, suppose that the stable square-integrable representations $\pi$ of $G(E)$ and $\pi'$ of $G'(E)$ are matching, then there exists an identity \[\dim \Hom_{G(F)}(\pi,\chi_{G,E})= \dim\Hom_{G'(F)}(\pi',\chi_{G',E}).  \]
\end{rem}
\par
Let us fix a element $\epsilon\in F^\times\backslash N_{E/F}E^\times$. Let $\SL_1(D)$ be the inner form of $\SL_2(F)$, which is not a quasi-split $F$-group, and there exists an embedding
\begin{equation}\label{matrixembedding}
\SL_1(D)=\Big\{g=\begin{pmatrix}
\bar{x}&\epsilon\bar{y}\\y&x
\end{pmatrix}|\det(g)=1,x,y\in E \Big\}\subset\SL_2(E) 
\end{equation}
where $\bar{x}=a-b\sqrt{\varpi}$ if $x=a+b\sqrt{\varpi}$ with $a,b\in F$ and $E=F[\sqrt{\varpi}]$, $\varpi\in F^\times\setminus{F^\times}^2$.
This paper will mainly discuss the distinction problem for $\SL_1(D)$ over a quadratic field extension $E/F$, i.e. the multiplicity
\[\dim\Hom_{\SL_1(D)}(\tau,\mathbb{C} ). \]
\par
Let $V_D$ be a $n$-dimensional Hermitian  $D$-vector space with Hermitian form $h$, then $$Aut(V_D,h)=\{g\in \GL_n(D)|h(gv_1,gv_2)=h(v_1,v_2),\forall v_1,v_2\in V \},$$
where $n=\dim_D V_D$. Assume that
\[h(v_1,v_2)=f(v_1,v_2)+B(v_1,v_2)j \]
where $D=E\oplus Ej, je=\bar{e}j$, $f(v_1,v_2)\in E$ and $B(v_1,v_2)\in E$ for $v_1,v_2\in V$.
Moverover, $f(v_1,v_2)=B(v_1,v_2j)$ and $f(v_1,v_2j)=B(v_1,v_2)j^2$, see \cite[\S10.3]{Scharlau1985}. Furthermore,
 regarding $V_D$ as a $2n$-dimensional vector space $\mathcal{V}_E$ over $E$, then $f$ (resp. $B$) is a non-degenerate Hermitian (resp. symplectic) form on $\mathcal{V}_E$ and
$Aut(\mathcal{V}_E,f)\cong U_{2n}(E/F)$ (resp. $Aut(\mathcal{V}_E,B)\cong\Sp_{2n}(E)$). Let $n=1$, we obtain a group embedding 
\begin{equation}\label{quaterembed}
\SL_1(D)=Aut(V_D,h)\hookrightarrow Aut(\mathcal{V}_E,B)=\SL_2(E) 
\end{equation}
which is the same as the embedding in \eqref{matrixembedding}. However the embedding \eqref{quaterembed} will be convenient for us when we use the local theta correspondence over the quaternionic unitary groups to deal with the distinction problem $\Hom_{\SL_1(D)}(\tau,\mathbb{C}) $.
\begin{thm}\label{localmain}
	Suppose that $\tau$ is an irreducible $\SL_1(D)$-distinguished representation of $\SL_2(E)$.
	\begin{enumerate}[(i).]
		\item If $\tau$ is a square-integrable representation, then  $$\dim\Hom_{\SL_1(D)}(\tau,\mathbb{C} )=\begin{cases}
		2,&\mbox{if }|\Pi_{\phi_\tau}|=2;\\
		1,&\mbox{otherwise}.
		\end{cases} $$
		Here $|\Pi_{\phi_\tau}|$ denotes the size of the $L$-packet $\Pi_{\phi_\tau}$.
		\item If $\tau=I(\chi|-|^z_E)$ is a principal series representation, then $\dim\Hom_{\SL_1(D)}(\tau,\mathbb{C} )=2$.
		\item If $\tau\subset I(\omega_{K/E})$, then $\dim\Hom_{\SL_1(D)}(\tau,\mathbb{C} )=1$.
	\end{enumerate}
\end{thm}
\par
Instead of considering each individual dimension, we consider the sum
\[S(\tau)=\sum_{\pi\in\Pi_{\phi_\tau} }\dim\Hom_{\SL_1(D)}(\pi,\mathbb{C})   \]
where $\Pi_{\phi_\tau}$ is the $L$-packet of representations of $\SL_2(E)$ containing  a $\SL_1(D)$-distinguished representation $\tau$.
\begin{thm}\label{thmforsum}
	Assume that $\tau$ is a $\SL_1(D)$-distinguished representation of $\SL_2(E)$ with an $L$-parameter $\phi_\tau$.
	\begin{enumerate}[(i).]
		\item
		Suppose that $\tau$ is a square-integrable representation.
		\begin{enumerate}[(a).]
			\item If $|\Pi_{\phi_\tau}|=1$, i.e. the size of the $L$-packet $\Pi_{\phi_\tau}$ is $1$, then $S(\tau)=1$.
			\item 	 If $|\Pi_{\phi_\tau}|=2$, then only one of them is $\SL_1(D)$-distinguished, the other is not $\SL_1(D)$-distinguished and
			$S(\tau)=2$.
			\item If $|\Pi_{\phi_\tau}|=4$ and $p\neq2$, then two members inside the $L$-packet $\Pi_{\phi_\tau}$ are $\SL_1(D)$-distinguished with the same multiplicity and $S(\tau)=2$.
			\item 	If $|\Pi_{\phi_\tau}|=4$ and $p=2$, then $S(\tau)=2$ or $4$.
		\end{enumerate}
		\item If $\tau$ is an irreducible principal series representation, then $S(\tau)=2$.
		\item If $\tau$ is not discrete but tempered  and $|\Pi_{\phi_\tau}|=2$, then $S(\tau)=2$.
	\end{enumerate}
\end{thm}
\par
We will use the local theta correspondence for the quaternionic groups to prove Theorem \ref{localmain}.
The basic ideas come from \cite{Lpacific,hengfei2017}. With the help of the explicit theta correspondences between  small groups, one can use the see-saw identities to transfer the disctinction problems for $\SL_1(D)$ to another side, which is related to the branching problems for the non-split torus and so it becomes easier, see $\S3$ for more details. 
\begin{rem}
	Anandavardhanan and Prasad \cite{prasadcomp} discuss the global period
	problems for $\SL_1(D)$ over a quadratic number field extension $\mathbb{E}/\mathbb{F}$. More precisely, \cite[Proposition 9.3]{prasadcomp} implies that there exists  
	an automorphic representation $\pi$ of $\SL_1(D)(\mathbb{A}_\mathbb{E})$ which is locally distinguished by $\SL_1(D)(\mathbb{A}_\mathbb{F})$, but not globally distinguished in terms of having nonzero period integral on this subgroup.
\end{rem}

Now we briefly describe the contents and the organization of this paper. In $\S2$, we
 set up the notation about the local theta lifts. In $\S3$, the proof of Theorem  \ref{localmain} will be given and then  Theorem \ref{thmforsum} follows immediately. Finally, we will give a table for the multiplicities in one $L$-packet $\Pi_{\phi_\tau}$ when $\tau$ is $\SL_1(D)$-distinguished.
\subsection*{Acknowledgments} The author thanks  Wee Teck Gan for his guidance and numerous discussions when he was doing his Ph.D. study at National University of Singapore. He would like to thank Dipendra Prasad for useful discussions as well, especially for his comments on earlier versions of this paper.
\input{localtheta}

\section{Proof of  Theorem \ref{localmain}}
Before we prove Theorem \ref{localmain},
let us recall some facts. Let $V_D$ denote the rank one Hermitian space over $D$ with quaternionic Hermitian group $U(V_D)=\SL_1(D)$.
\begin{lem}
	If the discriminant of $W_D=\mathfrak{V}_E\otimes_ED$ is nontrivial in $F^\times/{F^\times }^2$, say $L$, then the theta lift of the trivial representation from $\SL_1(D)$ to $U(W_D)=\GL_1(D_L)^\natural/F^\times$ is a character, i.e.
	\[\Theta_\psi(\mathbf{1})=\mathbf{1}\boxtimes\omega_{L/F}, \]
	where $D_L$ is a quaternion division algebra over $L$ and $\GL_1(D_L)^\natural=\{g\in D_L^\times| N_{D_L/L}(g)\in F^\times \}$.
\end{lem}
\begin{proof}
It follows directly from the  theta correspondence over the compact groups. More precisely,  following \cite[Proposition 5.1]{gan2014inner},
let $\mathbb{L}/\mathbb{F}$  be a quadratic extension of number fields and $\mathbb{D}$ (resp. $\mathbb{D}_\mathbb{L}$) a quaternion
$\mathbb{F}$-algebra (resp. $\mathbb{L}$-algebra) with involution $\ast$ such that for some place $v_0$ of $\mathbb{F}$, we have
\[(\mathbb{L}/\mathbb{F})_{v_0}=L/F~\mbox{  and  }\mathbb{D}_{v_0}=D ~~(\mbox{resp. }(\mathbb{D}_\mathbb{L})_{v_0}=D_{L}). \]
Let $\mathbb{V}$ denote the rank one Hermitian space over $\mathbb{D}$ with hermitian form
\[\langle x,y\rangle=x\cdot y^\ast \]
and let $\mathbb{W}$ denote the non-split rank $2$ skew-Hermitian space over $\mathbb{D}$ of discriminant $\mathbb{L}$, such that
\[\mathbb{V}_{v_0}=V_D\mbox{  and  }\mathbb{W}_{v_0}=W_D. \]
	Then one has a dual pair $U(\mathbb{V})\times U(\mathbb{W})$ over $\mathbb{F}$ 
and one may consider the global theta lift from
\[ U(\mathbb{V})=\SL_1(\mathbb{D}) \]
to 
\[U(\mathbb{W})=\GL_1( \mathbb{D}_{\mathbb{L}})^\natural/\mathbb{F}^\times \]
where $\GL_1( \mathbb{D}_{\mathbb{L}})^\natural=\{g\in \mathbb{D}^\times_\mathbb{L}:N_{\mathbb{D}_\mathbb{L}/\mathbb{L}}(g)\in\mathbb{F}^\times \}$. The global theta lift to $U(\mathbb{W})$ of trivial representation of $\SL_1(\mathbb{D})$ is nonzero  since we are in the stable range.
Moreover, at the places where $\mathbb{D}$ is unramified, \cite[Lemma 3.1]{Lpacific} implies that the local theta lift of the trivial representation  is a character of $U(\mathbb{W}_v)$.
By the strong multiplicity one theorem for $\GL_1(\mathbb{D}_\mathbb{L})$, we conclude that
\[\Theta(\mathbf{1})=\mathbf{1}\boxtimes\omega_{\mathbb{L}/\mathbb{F}}. \]
By the local-global compatibility of theta correspondence, we have $\theta_{\psi }(\mathbf{1})=\mathbf{1}\boxtimes\omega_{L/F}$. Because $U(W_D)$ is a compact group, Howe duality theorem implies that
\[\Theta_\psi(\mathbf{1})=\theta_{\psi }(\mathbf{1})=\mathbf{1}\boxtimes\omega_{L/F}. \]
Then we are done.
\end{proof}

Now we start to prove Theorem \ref{localmain}.
\subsection*{Proof of Theorem \ref{localmain}}
We separate the proof  into four cases as follows:
\begin{itemize}
	\item $\tau$ is a supercuspidal representation, see (A);
	\item  $\tau$ is an irreducible principal series representation, see (B);
	\item $\tau$ is a Steinberg representation $St_E$, see (C);
	\item $\tau$ is a constituent of a reducible principle series $I(\chi)$ with $\chi^2=1,$ see (D).
\end{itemize}
\begin{enumerate}[(A).]
	\item 
If $\tau$ is supercuspidal, then 
there exists a character $\mu:K^\times\rightarrow\mathbb{C}^\times$
such that $\phi_\tau=i\circ (Ind_{W_K}^{W_E}\mu)$, where 
\begin{itemize}
	\item  $W_K$ is the Weil group of $K$, which is a quadratic field extension over $E$;
	\item $\mu$
	does not factor through the norm map $N_{K/E},$ so
	 the irreducible Langlands parameter $$Ind_{W_K}^{W_E}\mu:W_E\rightarrow\GL_2(\mathbb{C})$$ corresponds to a dihedral supercuspidal representation of $\GL_2(E)$ with respect to $K$;
	 \item $i:\GL_2(\mathbb{C})\rightarrow \PGL_2(\mathbb{C})$ is the projection map, which coincides with the adjoint map $$Ad:\GL(2)\rightarrow\SO(3).$$
\end{itemize}
 In fact, if $\tau=\theta_{\psi }(\Sigma)$,  where $\Sigma$ is a representation of $\Oo(\mathfrak{V}_E)$ and $\mathfrak{V}_E$ is a $2$-dimensional $E$-vector space of discriminant $K$, then the Langlands parameter $\phi$ of $\Sigma$ is given by
 \[\phi(g)=\begin{cases}
 \begin{pmatrix}
 \chi_K(g)\\&\chi_K^{-1}(g)
 \end{pmatrix}&\mbox{if } g\in W_K\\
 \begin{pmatrix}
 0&1\\1&0
 \end{pmatrix}&\mbox{if }g=s
 \end{cases} \]
 where $s\in W_E\setminus W_K$ and the character $\chi_K:W_K\rightarrow\mathbb{C}^\times$ is the pull back of a nontrivial character $\mu_1$ of $K^1$ under the  map $K^\times\rightarrow K^1$ via $k\mapsto k^sk^{-1}$, i.e. $\chi_K(k)=\mu_1(k^sk^{-1})$, see \cite[\S6.4]{kudla1996notes}. Furthermore, there is an isomorphism between two Langlands parameters of $\Oo(2)$
 \[\phi\otimes\omega_{K/E}\cong Ind_{W_K}^{W_E}\frac{\mu^s}{\mu}. \]
 In other words, one has $\chi_K=\mu^s\mu^{-1}$ and $\mu_1=\mu|_{K^1}$ is the restricted character.

Moreover, if $\mu_1^2\neq\mathbf{1}$, then $\tau=\theta_\psi(Ind_{\SO(\mathfrak{V}_E)}^{\Oo(\mathfrak{V}_E)}(\mu_1) )$.
If $\mu_1^2=\mathbf{1}$, then there are two extensions of $\mu_1$ from $\SO(\mathfrak{V}_E)$ to $\Oo(\mathfrak{V}_E)$, denoted by $\mu_1^\pm$. The theta lift of $\mu_1^+$ (resp. $\mu_1^-$) from $\Oo(\mathfrak{V}_E)$ to $\SL_2(E)$ is a tempered representation $\tau^+$ (resp. $\tau^-$).  For convenience, if $\mu_1^2\neq\mathbf{1}$, we use $\mu^+=\mu^-$ to denote $Ind_{\SO(\mathfrak{V}_E)}^{\Oo(\mathfrak{V}_E)}\mu_1$ as well.
Assume that
$\Theta_\psi(\mu_1^+ )$ is a supercuspidal representation of $\SL_2(E)$.
\par
If the discriminant $disc (\mathfrak{V}_E\otimes_ED)\in F^\times/(F^\times )^2$ is nontrivial,  by the see-saw diagram
\[\xymatrix{\tau^+\oplus\tau^-& \SL_2(E)\ar@{-}[rd] & U(W_D)\ar@{-}[ld] &\Theta_\psi(\mathbf{1} )\\ \mathbf{1}&\SL_1(D)& \SO(\mathfrak{V}_E)&{\mu_1} } \]
where $\tau^-=0$ if $\mu_1^2\neq\mathbf{1}$, one has an isomorphism
\[\Hom_{\SL_1(D)}(\tau^+\oplus\tau^-,\mathbb{C} )\cong \Hom_{\SO(\mathfrak{V}_E)}(\mathbf{1},\mu_1 ) \]
which is nonzero if and only if $\mu_1=\mathbf{1}.$ But $\Hom_{K^1}(\mathbf{1},\mu_1)=0$, then
 $\Hom_{\SL_1(D)}(\tau^\pm,\mathbb{C})=0$.
\par
If the discriminant of $\mathfrak{V}_E\otimes_ED$ is $1\in F^\times/(F^\times)^2 $, 
 we denote by $\mathcal{I}(s)$ the degenerate principal series of $U_{1,1}(D)$ and we assume that $F^\times/(F^\times)^2\supset \{1,u,\varpi,u\varpi \}$ and $E=F[\sqrt{\varpi} ]$ with associated Galois group $\Gal(E/F)=\langle\sigma\rangle$, $K=E[\sqrt{u}]$, then \eqref{pmseesaw} implies
\begin{equation}\label{openorbit}
\Hom_{\SL_1(D) }(\tau^+,\mathbb{C} )=\Hom_{\Oo(\mathfrak{V}_E)}(\mathcal{I}(\frac{1}{2} ),\mu_1^- )\cong  \Hom_{U(W')}(({\mu_1^-})^{-1},\mathbb{C} ) 
\end{equation} 
where $K$ is a quadratic unramified extension  over $E$,  $W'$   is a one-dimensional skew-Hermitian $D$-vector space  with discriminant $u$. 
In this case,  \eqref{openorbit} can be rewritten as the following identity
\begin{equation}\label{thesum}
\dim \Hom_{\SL_1(D)}(\tau^+,\mathbb{C} )=\dim \Hom_{U(W')}(\mu_1^-,\mathbb{C} ) 
\end{equation}
which is nonzero if and only if 
	 \begin{eqnarray} \label{omega}
	 \mu(x-y\sqrt{u})=\mu(x+y\sqrt{u})
	 \end{eqnarray} for $x,y\in F$. Similarly, if $\mu_1^2=\mathbf{1}$, one has
	 \[\dim\Hom_{\SL_1(D)}(\tau^-,\mathbb{C})=\dim\Hom_{U(W')}(\mu_1^+,\mathbb{C} ).  \]
\begin{rem} If the Hasse-invariant of $\mathfrak{V}_E$ is $-1$ and the discriminant of $\mathfrak{V}_E$ is $K$, then $$\dim\Hom_{\SL_1(D)}(\tau^+,\mathbb{C})\neq0$$ if and only if 
	\begin{equation}\label{uomega}
	\mu(x-y\sqrt{u\varpi})=\mu(x+y\sqrt{u\varpi}) 
	\end{equation} for $x,y\in F$. 
	Because $\mu^s\neq\mu$, \eqref{omega} and \eqref{uomega} can not hold at the same time unless $p=2.$
\end{rem}
\begin{lem}\label{siegelweil}
	Let $\mathfrak{V}_E$ be a $2$-dimensional quadratic $E$-vector space. Assume that $W_D=\mathfrak{V}_E\otimes_ED$ is a $2$-dimensional skew-Hermitian $D$-vector space of trivial discriminant and $\pi$ is an irreducible representation of $\Oo(\mathfrak{V}_E)$, then
	\begin{equation}\label{pmseesaw}
	\dim\Hom_{\SL_1(D)}(\Theta_\psi(\pi\otimes\det),\mathbb{C} )=\dim\Hom_{\Oo(\mathfrak{V}_E)}(\mathcal{I}(\frac{1}{2}),\pi ).
	\end{equation}
	where the big theta lift $\Theta_\psi(\pi\otimes\det)$ is under the splitting $i_1:\SL_2(E)\times\Oo(\mathfrak{V}_E)\rightarrow\Mp_8(F)$.
\end{lem}
\begin{proof}
	Let us fix the splitting $i_2:\SL_1(D)\times U(W_D)\rightarrow\Mp_8(F) $, then \cite[Theorem 1.3]{yamana2011deg} implies that $\Theta_\psi(\mathbf{1})=\mathcal{I}(\frac{1}{2})$ is an irreducible representation of $U(W_D)$. The splitting from $\SL_2(E)$ to $\Mp_8(F)$ is unique, so $i_1 i_2^{-1}$ is a quadratic character on $\SL_1(D)\times\Oo(\mathfrak{V}_E)$ and trivial on $\SL_1(D)$. Thus,
	\begin{eqnarray}
	\begin{split}
	\dim\Hom_{\SL_1(D)}(\Theta_\psi(\pi\otimes\det),\mathbb{C})&
	=\dim\Hom_{i_1(\SL_1(D)\times \Oo(\mathfrak{V}_E))}(\omega_\psi,\widetilde{\mathbb{C}}\otimes\widetilde{(\pi\otimes\det)})\\
	&=\dim\Hom_{i_2(\SL_1(D)\times\Oo(\mathfrak{V}_E))}(\omega_\psi,\widetilde{\mathbb{C}}\otimes\widetilde{\pi} )\\
	&=\dim\Hom_{\Oo(\mathfrak{V}_E)}(\Theta_\psi(\mathbf{1}),\pi) \\
	&=\dim\Hom_{\Oo(\mathfrak{V}_E)}(\mathcal{I}(\frac{1}{2}),\pi ). 
	\end{split}
	\end{eqnarray}
	where $\widetilde{\pi}(h,\epsilon)=\epsilon\cdot\pi(h)$ for $(h,\epsilon)\in\widetilde{\Oo(W)}$.
\end{proof}
	If $p\neq2
	$ and $\mu_1^2=\mathbf{1}$,  \eqref{thesum} implies that
		 if  $\mu|_{E'}$ factors through the norm map $N_{E'/F}$ for $E'\neq E$, there exists a supercuspidal representation $\tau$ of $\SL_2(E)$ distinguished by $\SL_1(D)$. Then $$ \dim\Hom_{\SL_1(D)} (\tau^+,\mathbb{C})=1=\dim\Hom_{\SL_1(D)}(\tau^-,\mathbb{C} ) .$$
		 In the $L$-packet containing a $\SL_1(D)$-distinguished representation $\tau$, half members in $\Pi_{\phi_\tau}$ are $\SL_1(D)$-distinguished and
		 \[\sum_{\tau'\in\Pi_{\phi_\tau}}\dim \Hom_{\SL_1(D)}(\tau',\mathbb{C} )=2. \] 
		 \par
		 If $p\neq2$ and $\mu_1^2\neq\mathbf{1}$, then $\dim \Hom_{\SL_1(D)}(\tau,\mathbb{C}) =\dim\Hom_{\Oo(\mathfrak{V}_E)}(\mathcal{I}(\frac{1}{2}),\mu_1^+)$
		 which equals to the sum $$\dim\Hom_{U(W')}(\mu_1,\mathbb{C})+\dim\Hom_{U(W')}(\mu_1^{-1},\mathbb{C} )=\begin{cases}
		 2,&\mbox{if }\mu|_{E'}=\chi_F\circ N_{E'/F},E'\neq E;\\
		 0,&\mbox{otherwise }. 
		 \end{cases} $$

\par
If $p=2$, there are two more cases.
\begin{enumerate}[(i).]
	\item  Suppose that there are two distinct quadratic fields $E'$ and $E''$ over $F$ such that $\mu|_{E'}=\chi_F'\circ N_{E'/F}$ and $\mu|_{E''}=\chi_F''\circ N_{E''/F}$. Furthermore,  $\frac{\chi_F'}{\chi_F''}$ is a quadratic character of $F^\times$ that is not trivial restricted on the Weil group $W_K$ of $K$, i.e. $\frac{\chi_F'}{\chi_F''}$ is different from three quadratic characters $\omega_{E/F}$, $\omega_{E'/F}$ and $\omega_{E''/F}$, 
	\[\mu(t)=\mu^s(t)\cdot\frac{\chi_F'}{\chi_F''}\Big|_{W_K}(t),~t\in W_K \]
	 which may happen only when $p=2$. In this case, we obtain 
$\dim \Hom_{\SL_1(D)}(\tau^+,\mathbb{C})=1$ by the identity \eqref{thesum}. 
 Suppose that $\tau$ is $\SL_1(D)$-distinguished, then the set $\{\dim\Hom_{\SL_1(D)}(\tau',\mathbb{C}) :\tau'\in\Pi_{\phi_{\tau^+}} \}$ is $\{1,1,1,1 \}$ and
 $$\sum_{\tau'\in\Pi_{\phi_\tau}}\dim\Hom_{\SL_1(D)}(\tau',\mathbb{C})=4. $$
 \begin{rem}
 	For the $\SL_2(F)$-distinction problem, the set of the multiplicities in the $L$-packet $\Pi_{\phi_\tau}$ is $\{4,0,0,0 \}$ in this case, see \cite{anandavardhanan2003distinguished,Lpacific}.
 \end{rem}
	\item A cuspidal representation $\pi$ of $\GL_2(E)$, which is not dihedral with respect to any quadratic extension $K$ over $E$,
	 is irreducible when restricted to $\SL_2(E)$. Suppose that $\tau=\pi|_{\SL_2(E)}$ is irreducible.
	\par
	We consider a $2$-dimensional skew-Hermitian $D$-vector space $X$  with trivial discriminant,  then  $U(X)=U_{1,1}(D)$ can be naturally embedded into the special orthogonal group $\SO(2,2)(E)$. Let $\pi\boxtimes\pi$ be the irreducible representation of the similitude special orthogonal group $\GSO(2,2)(E)$. By the property of the big theta lift $\Theta(\pi)$ from $\GL_2(E)$ to $\GSO(2,2)(E)$,  $$(\pi\boxtimes\pi)|_{\SO(2,2)(E)}=\Theta(\pi)|_{\SO(2,2)(E)}=\Theta(\pi|_{\SL_2(E)})=\Theta(\tau)$$ is irreducible since $\tau$ is supercuspidal. Suppose that $Y$ is a $2$-dimensional Hermitian $D$-vector space. Let $\mathfrak{I}(s)$ be the degenerate principal series of $U(Y)$. 
	Considering  the following see-saw diagram
	\[\xymatrix{\mathfrak{I}(\frac{1}{2})&U(Y)\ar@{-}[rd] & \SO(2,2)(E)&(\pi\boxtimes\pi
		)\\\pi|_{\SL_2(E)} &\SL_2(E)\ar@{-}[ru] &U_{1,1}(D)&\mathbb{C} } \]
	due to the structure of $\mathfrak{I}(\frac{1}{2})$  in
	\cite[Theorem 1.4]{yamana2011deg}, one can get an equality
	\[\dim \Hom_{\SL_2(E)}(\mathfrak{I}(\frac{1}{2}),\pi )=\dim \Hom_{U_{1,1}(D) }((\pi\boxtimes\pi
	)|_{\SO(2,2)(E)},\mathbb{C}). \]
	The supercuspidal representation $\pi|_{\SL_2(E)}$ does not occur on the boundary of $\mathfrak{I}(\frac{1}{2}),$ then
	\[\dim \Hom_{\SL_2(E) }(\mathfrak{I}(\frac{1}{2}),\pi )=\dim \Hom_{\SL_1(D)}(\pi^\vee,\mathbb{C} ). \]
	Hence 
	\begin{equation}
	\begin{split}
&	\dim \Hom_{\SL_1(D)}(\pi^\vee,\mathbb{C} )\\
	=&\dim\Hom_{U_{1,1}(D) }((\pi\boxtimes\pi)|_{\SO(2,2)(E)},\mathbb{C})\\
	=&\dim\Hom_{GU_{1,1}(D)}(\pi\boxtimes\pi,\mathbb{C} )+\dim\Hom_{GU_{1,1}(D)}(\pi\boxtimes\pi,\omega_{E/F}) \\
	=&\dim \Hom_{\GL_2(F)}(\pi,\mathbb{C} )\dim\Hom_{D^\times}(\pi,\mathbb{C})+\dim\Hom_{\GL_2(F)}(\pi,\omega_{E/F})\dim\Hom_{D^\times}(\pi,\omega_{E/F})  .
	\end{split}
	\end{equation}
	Therefore, if $\pi$ is not dihedral with respect to any quadratic field extension $K$ over $E$ and so $\tau=\pi|_{\SL_2(E)}$ is irreducible, then the following are equivalent:
	\begin{itemize}
		\item the Langlands parameter $\phi_\pi$ is  conjugate-self-dual in the sense of \cite[\S 3]{gan2011symplectic}; 
		\item $\dim \Hom_{\SL_1(D)}(\tau,\mathbb{C})=1.$
	\end{itemize}
\end{enumerate}
	\begin{rem}
	This method can be used to deal with the case when $\tau$ is the Steinberg representation $St_E$ of $\SL_2(E)$, which will imply  $\dim\Hom_{\SL_1(D)}(St_E,\mathbb{C})=1$ directly. 
\end{rem}
	\item 
Let $\chi$ be a unitary character of $E^\times$. Since there is only one orbit for $D^\times$-action on the projective variety $P(E)\backslash\GL_2(E)\cong B(E)\backslash \SL_2(E)$, where $P(E)$ is the Borel subgroup of $\GL_2(E)$, its stabilizer is isomorphic to $E^\times$ and $B(E)\backslash\SL_2(E)\cong E^\times\backslash D^\times$. There are two orbits for $\SL_1(D)$-action on $B(E)\backslash \SL_2(E)$.
If $\tau=I(z,\chi)=Ind_{B(E)}^{\SL_2(E) }\chi|-|_E^z$ (normalized induction) is an irreducible principal series, due to  the double coset decomposition
$$\SL_2(E)=B(E)\SL_1(D)\sqcup B(E)\eta \SL_1(D),$$
where $\eta=\begin{pmatrix}
z_1&\bar{z}_2\\z_2&\bar{z}_1/\epsilon
\end{pmatrix}$, $d=z_1+z_2j,z_1,z_2\in E$ and $N_{D/F}(d)=\epsilon\in F^\times\setminus N_{E/F}E^\times$,
 then there is an exact sequence
\begin{equation}\label{doublecoset}
\xymatrix{ 0\ar[r]&\Hom_{E^1}(\chi,\mathbb{C})\ar[r]& \Hom_{\SL_1(D)}(\tau,\mathbb{C} ) \ar[r]& \Hom_{E^1}(\chi,\mathbb{C} )\ar[r]&0 }
\end{equation}
where  $ E^1=\ker N_{E/F}$. 
Then $\dim\Hom_{\SL_1(D)}(\tau,\mathbb{C})=2 $ if and only if 
	 $\chi=\chi_F\circ N_{E/F} $.
\item
If $\tau=St_E$ is a Steinberg representation of $\SL_2(E)$, then  the exact sequence \eqref{doublecoset} implies that
\[\dim\Hom_{\SL_1(D)}(I(|-|_E),\mathbb{C})=2,\]  so that  $\dim \Hom_{\SL_1(D) }(St_E,\mathbb{C} )=2-1=1$. 
\item Assume that $\tau$ is tempered.
If $\tau\subset I(\omega_{K/E})$ is an irreducible constituent of a reducible principal series,  set $\chi=\omega_{K/E},~\chi^+(\omega)=1,\omega=\begin{pmatrix}
&1\\1
\end{pmatrix},$ then from \cite[Page 86]{kudla1996notes}, we can see that
\[I(\omega_{K/E} )=\theta_\psi(\chi^+ )\oplus\theta_\psi(\chi^-)\mbox{ where }\chi^-= \chi^+\otimes\det  \]
and $\tau^+=\theta_\psi(\chi^+)=\Theta_\psi(\chi^+),\tau^-=\theta_\psi(\chi^-)$, where  $\theta_\psi(\chi^\pm)$ is the theta lift of $\chi^\pm$ from
$\Oo_{1,1}(E)$ to $\SL_2(E).$ By \eqref{pmseesaw} and the see-saw diagram
\[\xymatrix{\tau^+& \SL_2(E)\ar@{-}[rd]&U_{1,1}(D)\ar@{-}[ld]&\mathcal{I}(1/2)\\\mathbb{C}&\SL_1(D)&\Oo_{1,1}(E)&\chi^+\otimes\det } \]
where $\mathcal{I}(s)$ is the principal series of $U_{1,1}(D)$,   one has an identity
\[\dim \Hom_{\SL_1(D) }(\tau^+,\mathbb{C} )=\dim \Hom_{\Oo_{1,1}(E)}(\mathcal{I}(\frac{1}{2} ),\chi^+\otimes\det) \]
which is equal to $$ \dim \Hom_{E^1}(\chi,\mathbb{C} )=\begin{cases}
1&\mbox{if }\chi=\chi_F\circ N_{E/F};\\0&\mbox{otherwise.}
\end{cases}$$
Similarly, one can prove $\dim\Hom_{\SL_1(D) }(\tau^+,\mathbb{C})=\dim\Hom_{\SL_1(D) }(\tau^-,\mathbb{C})$.
\end{enumerate}
Then we finish the proof of Theorem \ref{localmain}.
\begin{coro} Let $\tau$ be a $\SL_1(D)$-distinguished representation of $\SL_2(E)$.
	If the representation $\tau'$ lies in the $L$-packet $\Pi_{\phi_\tau}$, then $\dim\Hom_{\SL_1(D) }(\tau',\mathbb{C})$ is either $0$ or $\dim\Hom_{\SL_1(D) }(\tau,\mathbb{C})$.
\end{coro}
In fact, Theorem \ref{thmforsum} follows from the above arguments as well.

Let us display the results of the multiplicities for the $L$-packet $\Pi_{\phi_\tau}$ containing a $\SL_1(D)$-distinguished or $\SL_2(F)$-distinguished representation $\tau$ of $\SL_2(E)$ in the form of tables.
\par 
\begin{table}[h]
	\begin{center}
		\renewcommand*{\arraystretch}{1.5}
		\caption{Assume $\tau$ is square-integrable }
		\begin{tabular}{|c|c|c|c|}
			\hline 
			&$\SL_1(D)$-distinguished&$\SL_2(F)$-distinguished& the character $\mu$ of $K^\times,K=E[\sqrt{u}]$  \\ \hline
			$|\Pi_{\phi_\tau}|=1$&$\{1\}$&$\{1\} $& N.A. \\ \hline
			$|\Pi_{\phi_\tau}|=2$& $\{2,0 \}$&$\{2,0 \} $&$\mu(x-y\sqrt{u})=\mu(x+y\sqrt{u}),\mu_1^2\neq\mathbf{1}$ \\ \hline
			$|\Pi_{\phi_\tau}|=4$& $\{1,1,0,0\}$&$\{1,1,0,0 \}$&$\mu^s=-\mu$ and $\mu_1^2=\mathbf{1}$  \\ \cline{2-4}
			&$\{1,1,1,1 \}$&$\{4,0,0,0 \} $&$\mu^s=\mu\chi_F'/\chi_F''\neq\pm\mu$ and $\mu_1^2=\mathbf{1}$
			\\ \hline
		\end{tabular}
	\end{center}
\end{table}

The last row occurs only when $p=2$ and $\tau$ is a supercuspidal representation of $\SL_2(E)$.

\begin{table}[h]
\begin{center}
	\renewcommand*{\arraystretch}{1.5}
	\caption{Assume $\tau$ is not square-integrable}
	\begin{tabular}{|c|c|c|c|p{3cm}}
		\hline &$\SL_1(D)$-distinguished&$\SL_2(F)$-distinguished&the character $\chi_E$ of $E^\times$ \\ \hline
		$|\Pi_{\phi_\tau}|=1$& $\{2\}$&$\{2\} $&$\chi_E=\mathbf{1}$ \\ \cline{2-4}
		&$\{2\}$&$\{2\} $& $\chi_E=\chi_F\circ N_{E/F}$ and $\chi_E^2\neq\mathbf{1}$ \\ \cline{2-4}
		&$\{ 0\}$&$\{1\}$& $\chi_E|_{F^\times}=\mathbf{1}$ and $\chi_E^2\neq\mathbf{1}$
		 \\ \hline
		$|\Pi_{\phi_\tau}|=2$& $\{1,1\}$&$\{1,1 \}$&$\omega_{K/E}=\chi_F\circ N_{E/F}$ with $\chi_F^2=\omega_{E/F}$ \\ \cline{2-4}
		&$\{1,1 \}$&$\{3,0 \} $&$\omega_{K/E}=\chi_F\circ N_{E/F}$ with $\chi_F^2=\mathbf{1}$.
		\\ \hline
	\end{tabular}
\end{center}
\end{table}
\begin{rem}
	If $\tau=I(\chi_E)$ is an irreducible principle representation of $\SL_2(E)$, where $\chi_E$ is a unitary character of $E^\times$ with $\chi_E^2\neq\mathbf{1}$ and $\chi_E|_{F^\times}=\mathbf{1}$, then $I(\chi_E)$ is not $\SL_1(D)$-distinguished but $\SL_2(F)$-distinguished. It corresponds to the case that  the  representation
	$\pi=\pi(\chi,\chi\chi_E)$ of $\GL_2(E)$ with  $\chi|_{F^\times}=\mathbf{1}$, is not $\GL_1(D)$-distinguished but $\GL_2(F)$-distinguished. 
\end{rem}
\begin{rem}
	Assume that $\tau\subset I(\omega_{K/E})$, where $K$ is a quadratic field extension over $E$ associated with a quadratic character $\omega_{K/E}$ by the local class field theory. If $\omega_{K/E}|_{F^\times}=\mathbf{1}$, then $\omega_{K/E}^\sigma=\omega_{K/E}$ and so $\omega_{K/E}$ must factor through the norm map $N_{E/F}$. The third case of D from \cite[Page 490]{Lpacific} does not exist, i.e. the set $\{1,0\}$ does not appear in the above tables when $\tau$ is $\SL_2(F)$-distinguished. 
\end{rem}

\bibliographystyle{amsalpha}
\bibliography{SL(D)}
\end{document}

%% file: localtheta.tex
\section{The Local Theta Correspondences }
In this section, we will briefly recall some results about the local theta correspondence, following \cite{moeglin1987correspondances}.

Let $F$ be a local field of characteristic zero.
Consider the dual pair $\Oo(m)\times \Sp(2n)$.
For simplicity, we may assume that $m$ is even. Fix a nontrivial additive character $\psi$ of $F$.
Let $\omega_\psi$ be the Weil representation for $\Oo(m)\times \Sp(2n)$. 
If $\pi$ is an irreducible representation of $\Oo(m)$ (resp. $\Sp(2n)$), the maximal $\pi$-isotypic quotient of $\omega_\psi$ has the form 
\[\pi\boxtimes\Theta_\psi(\pi) \]
for some smooth representation $\Theta_\psi(\pi)$ of $\Sp(2n)$ (resp. $\Oo(m)$). We call $\Theta_\psi(\pi )$
the big theta lift of $\pi.$ It is known that $\Theta_\psi(\pi)$ is of finite length and hence is admissible. Let $\theta_\psi(\pi)$ be the maximal semisimple quotient of $\Theta_\psi(\pi),$ which is called the small theta lift of $\pi.$ Then it was  conjectured by Roger Howe  that
\begin{itemize}
	\item $\theta_\psi(\pi)$ is irreducible whenever $\Theta_\psi(\pi)$ is non-zero.
	\item the map $\pi\mapsto \theta_\psi(\pi)$ is injective on its domain.
\end{itemize}
This has been proved by Waldspurger \cite{waldspurger1990demonstration} when the residual characteristic $p$ of $F$ is is not $2.$ Gan and Takeda \cite{gan2014howe,gan2014proof} have proved it completely.
\begin{thm}[Howe duality conjecture]
The	 Howe conjecture holds.
\end{thm}
 Gan and Sun \cite{gan2015howe}  prove the Howe duality conjecture for the quaternionic unitary groups.
More precisely,
 Let $D$ be the $4$-dimensional quaternion division algebra over $F$. Let $V_D$ be an $n$-dimensional Hermitian $D$-vector space with quaternionic Hermitian $F$-group $U(V_D)$. Let $W_D$ be an $m$-dimensional skew-Hermitian $D$-vector space with quaternionic Hermitian $F$-group $U(W_D)$, then there is an embedding of $F$-groups
\[U(V_D)\times U(W_D)\longrightarrow\Sp(4mn). \] 
One may define the Weil representation $\omega_\psi$ on $U(V_D)\times U(W_D)$ similarly. Given an irreducible representation $\pi$ of $U(V_D)$ (resp. $U(W_D)$), the maximal $\pi$-isotypic quotient of 
$\omega_\psi$ has the form $\pi\boxtimes\Theta_\psi(\pi)$ for some smooth representation $\Theta_\psi(\pi)$ of $U(W_D)$ (resp. $U(V_D)$), where $\Theta_\psi(\pi)$ is called the big theta lift and it has an irreducible quotient $\theta_{\psi }(\pi)$. And the map $\pi\mapsto\theta_\psi(\pi)$ is injective on its domain.

 \subsection{First occurence indices for pairs of orthogonal Witt towers} Let $W_n$ be the $2n$-dimensional symplectic vector space with associated symplectic group $\Sp(W_n)$ and consider the two towers of orthogonal groups attached to the quadratic spaces with nontrivial discriminant. More precisely, let
 $V_E$ (resp. $\epsilon V_E$) be the $2$-dimensional quadratic space with discriminant $E$ and Hasse invariant $+1$ (resp. $-1$), $\mathbb{H}$ be the $2$-dimensional hyperbolic quadratic space over $F$,
\[V_r^+=V_E\oplus \mathbb{H}^{r-1}\quad\mbox{and}\quad V_r^-=\epsilon V_E\oplus\mathbb{H}^{r-1} \]
and denote the orthogonal groups by $\Oo(V_r^+)$ and $\Oo(V_r^-)$ respectively. For an irreducible representation $\pi$ of $\Sp(W_n),$ one may consider the theta lifts $\theta^+_r(\pi)$ and $\theta^-_r(\pi)$ to
$\Oo(V^+_r)$ and $\Oo(V_r^-)$ respectively, with respect to a fixed non-trivial additive character $\psi.$ Set
\[\begin{cases}
r^+(\pi)=\inf\{2r:\theta^+_r(\pi)\neq0 \};\\
r^-(\pi)=\inf\{2r:\theta^-_r(\pi)\neq0 \}.
\end{cases} \]
Then Kudla and Rallis \cite{kudla2005first}, B. Sun and C. Zhu \cite{sun2012conservation} showed:
\begin{thm}
	[Conservation Relation] For any irreducible representation $\pi$ of $\Sp(W_n),$ we have
	\[r^+(\pi)+r^-(\pi)=4n+4=4+2\dim W_n. \]
\end{thm}
There is an analogous problem where one fixes an irreducible representation
of $\Oo(V_r^+)$ or $\Oo(V_r^-)$ and consider its theta lifts $\theta_n(\pi)$ to the tower of symplectic group $\Sp(W_n).$ Then with $n(\pi)$ defined in the analogous fashion, thanks to \cite[Theorem 1.10]{sun2012conservation}, we have
\[n(\pi)+n(\pi\otimes\det)=\dim V_r^\pm. \]
 \subsection{See-saw identities}
 Let $V_D$ be a Hermitian $D$-vector space with Hermitian form $h$ given by $$h(v_1,v_2)=f(v_1,v_2)+B(v_1,v_2)j,\forall v_1,v_2\in V_D.$$ Let $\mathfrak{W}_E=Res_{D/E}V_D$ be the  same space $V_D$ but now thought of as a $E$-vector space with a symplectic form $B$.
If $\mathfrak{V}_E$ is a quadratic $E$-vector space, then one can form a skew-Hermitian $D$-vector space $\mathfrak{V}_E\otimes_ED=W_D$ (which is not canonical, see \S2.3).
  Then we have the following isomorphism of symplectic spaces:
\[Res_{D/F }[W_D\otimes_D V_D ]\cong Res_{E/F} (\mathfrak{V}_E\otimes_E\mathfrak{W}_E )=\mathbf{W} \]
There is a pair 
\[(\Oo(\mathfrak{V}_E),\Sp(\mathfrak{W}_E) )\mbox{  and  }(U(W_D),U(V_D)) \]
of dual reductive pairs in the symplectic group $\Sp(\mathbf{W}).$
A pair $(G,H)$ and $(G',H')$ of dual reductive pairs in a symplectic group is called a see-saw pair if $H\subset G'$ and $H'\subset G.$ Let us fix the natural splittings \cite{kudla1994splitting}
\[ i_1:\Oo(\mathfrak{V}_E)\times \Sp(\mathfrak{W}_E)\longrightarrow\Mp(\mathbf{W}) \]
and $i_2:U(W_D)\times U(V_D)\longrightarrow \Mp(\mathbf{W})$.
\begin{lem}[See-saw identity]
	For a see-saw pair of dual reductive pairs $(\Sp(\mathfrak{W}_E),\Oo(\mathfrak{V}_E))$ and $(U(W_D),U(V_D))$, let $\pi$ be an irreducible representation of $\Oo(\mathfrak{V}_E)$ and $\pi'$ of $U(V_D),$ if the splittings $i_1$ and $i_2$ satisfy  \begin{equation}\label{seesawsplitting}
	i_1|_{\Oo(\mathfrak{V}_E)\times U(V_D)}=i_2|_{\Oo(\mathfrak{V}_E)\times U(V_D)}, \end{equation} 
	then we have the following isomorphism
	\[\Hom_{\Oo(\mathfrak{V}_E)}(\Theta_\psi(\pi'),\pi )\cong \Hom_{U(V_D)}(\Theta_\psi(\pi),\pi' ). \]
	\end{lem}
It follows from \cite[Theorem 8.2]{gurevich2015non}. However, \eqref{seesawsplitting} may not hold, see \cite[Lemma 4.3.7]{hengfei2017}. For our purpose, suppose $\dim_D V_D=1$ and $\dim_E \mathfrak{V}_E=2$,
 then
$U(V_D)\cong\SL_1(D)$. Let $\widetilde{\Oo(\mathfrak{V}_E)}$ denote the preimage of $\Oo(\mathfrak{V}_E)$ in $\Mp(\mathbf{W})$. Let $\widetilde{\pi}$ be a geniune representation of $\widetilde{\Oo(\mathfrak{V}_E)}$ associated to $\pi$, i.e.
\[\widetilde{\pi}(h,\epsilon)=\epsilon\cdot\pi(h)\mbox{  for  }(h,\epsilon)\in\widetilde{\Oo(\mathfrak{V}_E)}. \]
Observe that $i_1(h)=(h,1)\in\widetilde{\Oo(\mathfrak{V}_E)}$ and $i_2(h)=(h,\det(h))\in\widetilde{\Oo(\mathfrak{V}_E)}$
for $h\in\Oo(\mathfrak{V}_E)$. This means that $(i_1^{-1} i_2)|_{\Oo(\mathfrak{V}_E)}$ corresponds to the quadratic character $\det$ of $\Oo(\mathfrak{V}_E)$. Hence $$\Hom_{i_1(\Oo(\mathfrak{V}_E))}(\omega_\psi,\widetilde{\pi})\cong\Hom_{i_2(\Oo(\mathfrak{V}_E))}(\omega_\psi,\widetilde{\pi\otimes\det }).  $$
It will be useful in the proof of Theorem \ref{localmain}, see Lemma \ref{siegelweil}.
\subsection{Vector spaces}
Let $K/E$ be a quadratic field extension. Consider $K$ as a $2$-dimensional quadratic space $\mathfrak{V}_E$ over $E$ with the norm map $N_{K/E}.$ 
Given a $2$-dimensional quadratic $E$-vector space $\mathfrak{V}_E$ with a non-trivial discriminant $e\in E^\times\setminus{E^\times}^2$ and a  symmetic bilinear form $B$, set $$B(v_1,v_2j)=B(v_1,v_2)j\quad \mbox{and}\quad B(v_1d_1,v_2d_2)=B(v_2d_2,v_1d_1)$$
for $v_1,v_2\in \mathfrak{V}_E$ and $d_1,d_2\in D$. One may construct a skew-Hermitian form $h$ on $\mathfrak{V}_E\otimes_ED$
\[h(w_1,w_2)=B(w_1,w_2j)+B(w_1,w_2)j \]
where $w_1,w_2\in \mathfrak{V}_E\otimes_E D=W_D$. Then $W_D$ becomes to be a $2$-dimensional skew-Hermitian $D$-vector space with discriminant $N_{E/F}(e)\in F^\times/{F^\times}^2$, denoted by $W_D$. If $N_{E/F}(e)=1$, then $U(W_D)=U_{1,1}(D)$ and $K/E$ is an unramified field extension. If $N_{E/F}(e)$ is nontrivial, then the discriminant of $W_D$ corresponds to a quadratic field extension $L/F$. Moreover, there is a $4$-dimensional quaternion division algebra over $L$ such that 
\[U(W_D)=\GL_1(D_L)^\natural/F^\times \]
where $\GL_1(D_L)^\natural=\{x\in D_L^\times:N_{D_L/L}(x)\in F^\times \}$, see \cite[\S9]{prasad2011bessel}.
\begin{rem}
	If $\mathfrak{V}_E'=a\mathfrak{V}_E$ with a bilinear form $aB$, where $a\in E^\times\setminus N_{K/E}K^\times$, then the discriminant of the $D$-vector space $\mathfrak{V}_E'\otimes_ED$ does not change. So the skew-Hermitian $D$-vector space $W_D$ does not depend on the choice of the bilinear form $B$ on $\mathfrak{V}_E$. 
\end{rem}

\subsection{Degenerate principal series representations} Assume  $U(W_D)=U_{1,1}(D)$. There is a natural group embedding $\Oo(\mathfrak{V}_E)\hookrightarrow U_{1,1}(D)$. Let $P$ be a Siegel parabolic subgroup of $U_{1,1}(D).$ Assume that $\mathcal{I}(s)$ is the degenerate principal series of $U_{1,1}(D)$.
Let us consider the double coset decomposition $P\backslash U_{1,1}(D)/\Oo(\mathfrak{V}_E).$ 
\begin{itemize}
	\item If $K$ is a field, then there is only one orbit in $P\backslash U_{1,1}(D)/ \Oo(\mathfrak{V}_E)$.
	\item If $K=E\oplus E,$ then there is one open orbit and one closed orbit  in $P\backslash U_{1,1}(D)/ \Oo(\mathfrak{V}_E)$. 
\end{itemize}
Assume that there is a stratification $P\backslash U_{1,1}(D)/\Oo(\mathfrak{V}_E)=\sqcup_{i=0}^r X_i$ such that  $\sqcup_{i=0}^k X_i$ is open for each $k$ lying in $\{0,1,2,\cdots, r\}.$
Then
there is a $\Oo(\mathfrak{V}_E)$-equivariant filtration $\{I_i \}_{i=0,1,2,\cdots,r }$ of $\mathcal{I}(s)|_{\Oo(\mathfrak{V}_E)}$ such that 
$$0=I_{-1}\subset I_0\subset I_1\subset \cdots\subset I_r=\mathcal{I}(s)|_{\Oo(\mathfrak{V}_E)}$$
  and the smooth functions in the quotient $ I_{i}/I_{i-1}$ are supported on a single orbit $X_{i}$  in  $P\backslash U_{1,1}(D)/\Oo(\mathfrak{V}_E)$. 
\begin{defn} Given an irreducible representation $\pi$ of $\Oo(\mathfrak{V}_E)$, if $\Hom_{\Oo(\mathfrak{V}_E)}(I_{i+1}/I_i,\pi )\neq0$ implies that  $I_{i+1}/I_{i}$ is supported on the open orbits in $P\backslash U_{1,1}(D)/\Oo(\mathfrak{V}_E)$, then we say that the representation $\pi$ does not occur on the boundary of $\mathcal{I}(s).$
	\end{defn}
It is well-known that only the open orbits can support supercuspidal representations.